\documentclass[12pt]{article}
\usepackage{amssymb}
\usepackage{amsmath}
\usepackage{graphicx}
\usepackage{fullpage}
\usepackage{amsthm}
\usepackage{natbib}
\usepackage{float}
\bibpunct{[}{]}{,}{n}{}{;}

\newtheorem{theorem}{Theorem}
\numberwithin{theorem}{section}

\numberwithin{remark}{section}
\newtheorem{lemma}{Lemma}
\numberwithin{lemma}{section}
\newtheorem{corollary}{Corollary}
\numberwithin{corollary}{section}

\numberwithin{proposition}{section}
\newtheorem{conjecture}{Conjecture}
\numberwithin{conjecture}{section}
\numberwithin{equation}{section}

\newcommand{\inv}{\operatorname{inv}}
\newcommand{\maj}{\operatorname{maj}}
\newcommand{\area}{\operatorname{area}}
\newcommand{\dinv}{\operatorname{dinv}}
\newcommand{\ides}{\operatorname{ides}}

\newcommand{\comp}{\operatorname{comp}}
\newcommand{\touch}{\operatorname{touch}}

\newcommand{\diagword}{\operatorname{diagword}}
\newcommand{\Yconsec}{\operatorname{Yconsec}}

\title{A proof of the Square Paths Conjecture}
\author{Emily Sergel Leven}

\begin{document}

\maketitle

\begin{abstract}
The modified Macdonald polynomials, introduced by Garsia and Haiman (1996), have many astounding combinatorial properties. One such class of properties involves applying the related $\nabla$ operator of Bergeron and Garsia (1999) to basic symmetric functions. The first discovery of this type was the (recently proven) Shuffle Conjecture of Haglund, Haiman, Loehr, Remmel, and Ulyanov (2005), which relates the expression $\nabla e_n$ to parking functions. In (2007), Loehr and Warrington conjectured a similar expression for $\nabla p_n$ which is known as the Square Paths Conjecture.

Haglund and Loehr (2005) introduced the notion of schedules to enumerate parking functions with a fixed set of cars in each diagonal. In this paper, we extend the notion of schedules and some related results of Hicks (2013) to labeled square paths. We then apply our new results to prove the Square Paths Conjecture.
\end{abstract}


\section{Introduction}

This paper addresses the interplay between symmetric function theory and combinatorics. In particular, we prove that $\nabla p_n$ can be expressed as a weighted sum of certain labeled lattice paths (called preference functions or labeled square paths). This formula for $\nabla p_n$ was originally conjectured by Loehr and Warrington \citep{sqconj}. Here $p_n$ is the $n$th power symmetric function and $\nabla$ is the symmetric function operator introduced by Bergeron and Garsia \citep{SciFi}. This linear operator is defined by its action on the modified Macdonald polynomials ($\nabla$'s eigenfunctions). The Macdonald polynomials are a basis for the ring of symmetric functions first introduced by Macdonald \citep{macoriginal} and later modified by Garsia and Haiman \citep{modmac}.

The $\nabla$ operator is also a component of the Shuffle Conjecture. The symmetric function side of the Shuffle Conjecture - $\nabla$ applied to the elementary symmetric functions $e_n$ - was first studied because of its relation to the module of Diagonal Harmonics. In \citep{shuffleconj}, Haglund, Haiman, Loehr, Remmel, and Ulyanov conjectured a combinatorial formula for $\nabla e_n$ as an enumeration of certain labeled Dyck paths, called parking functions. This conjecture was refined by Haglund, Morse, and Zabrocki \citep{compconj} and their refinement was recently proved by Carlsson and Mellit \citep{compproof}.

These two classes of labeled lattice paths - parking functions and preference functions - are intimately related. Both were introduced by Konheim and Weiss \citep{KW} in 1966. A preference function is a map $f:[n] \to [n]$. For convenience, we will also write it as the vector $(f(1),f(2),\dots,f(n))$. A parking function is any preference function such that $|f^{-1}([k])| \geq k$ for all $1 \leq k \leq n$. Konheim and Weiss motivated this definition by describing a parking procedure in which $n$ cars try to park in $n$ spaces on a one-way street according to a preference function $f$. The cars will all succeed in parking if and only if the preference function is a parking function.

For our purposes, it is more helpful to think of the lattice-path interpretation of preference/parking functions. Start with an empty $n \times n$ lattice. Write each car which prefers spot 1 (each $i \in f^{-1}(1)$) in column 1, starting at the bottom, from smallest to largest. Then move to the lowest empty row and write all the cars which prefer spot 2 ($f^{-1}(2)$) in column 2 from smallest to largest and bottom to top. Continue this procedure until all the cars have been recorded. Then draw in the unique smallest lattice path which consists of North and East steps and stays above each car. For example, see Figure \ref{fig:pref}.

This gives a bijective correspondence between the $n^n$ preference functions and the set of North-East paths from $(0,0)$ to $(n,n)$ which (1) have column-increasing labels adjacent to North steps and (2) end with an East step. The underlying lattice paths here are also known as square paths and the labels are known as cars. Furthermore, such a labeled path corresponds to a parking function if and only if the underlying path stays (weakly) above the line $y=x$. The underlying paths here are known as Dyck paths.

\begin{figure}[H]
\begin{center}
\includegraphics[width=2.3in]{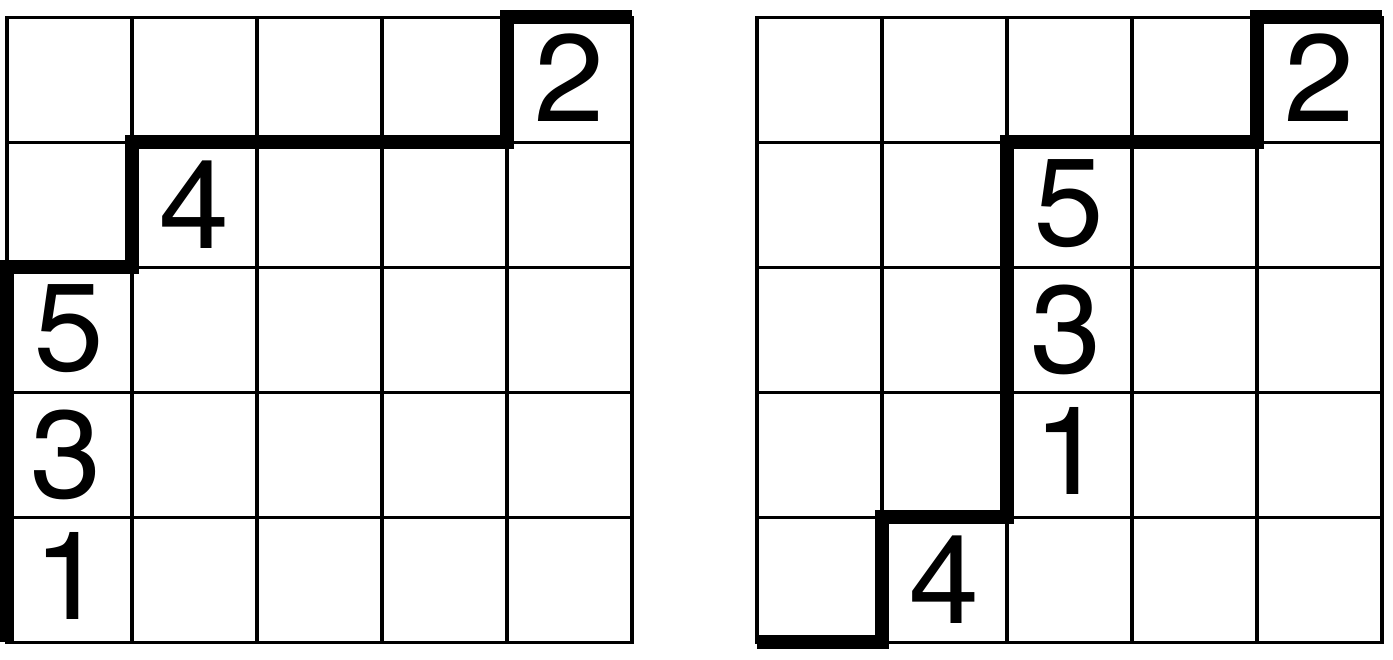}
\caption{The labeled paths corresponding to $(1,5,1,2,1)$ and $(3,5,3,2,3)$.}
\label{fig:pref}
\end{center}
\end{figure}
\vspace{-.5cm}

Since parking functions stay above the main diagonal $y=x$, a natural statistic is given by counting the number of full cells between the main diagonal and the underlying path. This statistic is known as the \emph{area}. The parking function in Figure \ref{fig:pref} (the path on the left) has $\area = 5$. The other statistics used in the Shuffle Conjecture are less natural. They make use of the diagonals of the parking function - those cells cut by a single line of the form $y=x+k$. One is the \emph{word} of a parking function $PF$, which is denoted $\sigma(PF)$. This is the permutation obtained by reading cars from highest to lowest diagonal and from right to left within each diagonal. The word of the parking function in Figure $\ref{fig:pref}$ is $4 \, 5 \, 3 \, 2 \, 1$. We will also make use of the \emph{descent set of the inverse} of $\sigma$. This is the set of $i$ for which $i+1$ comes before $i$ in $\sigma$. For simplicity, we will denote this by $\ides(PF)$. In Figure $\ref{fig:pref}$, the parking function has $\ides = \{ 1, 2, 3 \}$.

The final statistic is $\dinv$, whose name is an abbreviation of \emph{diagonal inversions}. There are two types of $\dinv$ for parking functions. A primary $\dinv$ occurs whenever two cars appear in the same diagonal and the car further left is smaller. In Figure \ref{fig:pref}, cars $1$ and $2$ form the only primary $\dinv$. A secondary $\dinv$ occurs whenever two cars appear in adjacent diagonals and the smaller car is both lower and further right. In Figure \ref{fig:pref}, cars $2$ and $3$ form the only secondary $\dinv$. The $\dinv$ of a parking function is the total number of primary and secondary $\dinv$s. Hence in our example, $\dinv=2$. 

Let ${\cal PF}_n$ be the set of all parking functions on $n$ cars. The original Shuffle Conjecture \citep{shuffleconj} states
$$
\nabla e_n = \sum_{PF \in {\cal PF}_n} t^{\area(PF)} q^{\dinv(PF)} Q_{\ides(PF)}.
$$
Here, for any $S \subseteq \{1,2,\dots,n-1\}$,
$$
Q_S(x_1,\dots,x_n) = \sum_{ \substack{ a_1 \leq \dots \leq a_n \\ i \in S \Rightarrow a_i < a_{i+1} } } x_{a_1} \cdots x_{a_n}
$$
is the fundamental quasi-symmetric function introduced by Gessel \citep{Gessel}. In Section \ref{sec:ides}, we will apply Haglund, Morse, and Zabrocki's refinement \citep{compconj} of the Shuffle Conjecture, which was recently proved by Carlsson and Mellit \citep{compproof}.

In \citep{sqconj}, Loehr and Warrington conjectured a similar formula for $\nabla p_n$. They express this symmetric function as an enumeration of all preference functions. Their statistics are similar to those used in the Shuffle Conjecture. The word of a preference function, for example, is calculated just as the word of a parking function is: the cars are read from highest to lowest diagonal and from right to left within each diagonal. We will again write $\ides(Pr)$ for the inverse descent set of the word of a preference function $Pr$. The preference function on the right of Figure \ref{fig:pref} has word $5 \, 2 \, 3 \, 1 \, 4$ and $\ides = \{1,4\}.$

The $\dinv$ of a preference function has three components: the usual primary and secondary $\dinv$s (within any diagonals) and a new component that we will call tertiary $\dinv$. The tertiary $\dinv$ is simply the number of cars strictly below the main diagonal $y=x$. For example, the preference function on the right of Figure \ref{fig:pref} has $\dinv = 3$. That is, it has no primary $\dinv$, one secondary $\dinv$ (between cars $2$ and $5$), and two tertiary $\dinv$s (contributed by cars $1$ and $4$).

To define the $\area$ of a preference function, we need to name diagonals. In particular, we will refer to the diagonal $y=x+k$ as the $k$-th diagonal. For any preference function $Pr$, let $l(Pr)$ be as large as possible so that the diagonal $y=x-l(Pr)$ is not empty. This is known as the deviation of the preference function. Note that $Pr$ is a parking function iff $l(Pr)=0$. Then $\area(Pr)$ is the sum over all cars of $Pr$ to which a car in diagonal $k$ contributes $k+l(Pr)$. So in the left side of Figure \ref{fig:pref}, the deviation is $1$ and $\area = 4$. 

It is easy to see that the two definitions given for $\dinv$ and word coincide when we view parking functions as (special) preference functions. To see the equivalence of the two definitions for the $\area$ of a parking function, note that a car in diagonal $k$ lies in a row with $k$ full cells between the underlying path and the main diagonal.

Let ${\cal P}ref_n$ be the set of all preference functions on $n$ cars.
\begin{conjecture}[Loehr-Warrington] \label{conj:sq}
$$
(-1)^{n-1} \nabla p_n = \sum_{Pr \in {\cal P}ref_n} t^{\area(Pr)} q^{\dinv(Pr)} F_{ides(\sigma(Pr))}
$$
\end{conjecture}

The main result of this paper is a proof of Conjecture \ref{conj:sq}. In Section \ref{sec:sched}, we extend a notion of Haglund and Loehr \citep{HL05} and use it to enumerate, by $\area$ and $\dinv$ alone, those preference functions with a fixed set of cars in each diagonal. In Section \ref{sec:shift}, we will discuss the effects of shifting cars between diagonals on the enumeration we obtained in Section \ref{sec:sched}. This will allow us to relate the enumeration of preference functions by $\area$ and $\dinv$ to the enumeration of parking functions by $\area$ and $\dinv$. Finally, in Section \ref{sec:ides}, we will show how to use the results of Section \ref{sec:shift} to relate the full enumerations (using $\area$, $\dinv$, and $\ides$) of preference and parking functions by extending a result of Hicks \cite{Athesis}. This, combined with a symmetric function identity and the Compositional Shuffle Conjecture, proves the Square Paths Conjecture.

In fact, we prove something much stronger: a relationship between the full enumerations of parking and preference functions with the same ``diagonal word" (which we introduce in the next section). This is analogous to Hicks' \citep{Athesis} conjecture that relations between different incarnations of the Compositional Shuffle Conjecture may be refined to the level of parking functions with fixed sets of cars in diagonals. This suggests that there may be quasi-symmetric refinements for the symmetric functions sides of the Shuffle Conjecture and Square Paths Conjecture which correspond to these combinatorial enumerations.\\

\noindent \textbf{Acknowledgements.} The author would like to thank Adriano Garsia and Jim Haglund for their insightful comments and discussions on this topic. This work was partially supported by NSF grant DGE-1144086.


\section{Schedules for preference functions} \label{sec:sched}

In this section we make heavy use of the diagonal word statistic and of the schedule of a parking function. These concepts were introduced by Haglund and Loehr in \citep{HL05} and expanded upon by Hicks in \citep{Athesis}. We follow the latter's notation.

The \emph{diagonal word} of a preference function $Pr$, denoted $\diagword(Pr)$, is a permutation whose runs give the cars in each diagonal of $Pr$ from highest to lowest diagonal. That is, cars from a single diagonal are listed in increasing order. This should not be confused with $Pr$'s word, $\sigma$, which lists cars from each diagonal in the order they actually appear. For example, the two preference functions in Figure \ref{fig:pref} have words $\sigma = 4 \, 5 \, 3 \, 2 \, 1$ and $\sigma = 5 \, 2 \, 3 \, 1 \, 4$, respectively, but diagonal words $4 \, 5 \, 3 \, 1 \, 2$ and $5 \, 2 \, 3\, 1 \, 4$.

This concept was first introduced to enumerate parking functions as follows. Let $\tau \in S_n$. Suppose the last run of $\tau$ has length $k$. Then for $1 \leq i \leq k$, let $w_{i} = i$. For $k < i \leq n$, let $w_i$ be the number of elements of $\tau_{n+1-i}$'s run which are larger than $\tau_{n+1-i}$ plus the number of cars smaller than $\tau_{n+1-i}$ in the next run. If $PF$ is a parking function with diagonal word $\tau$, then $W=(w_i)$ is called its schedule. We also say that $W$ is the schedule of $\tau$. There are $\prod_{i=1}^n w_i$ parking functions with diagonal word $\tau$ and they can be built by inserting the cars of $\tau$ from right to left into an empty parking function. 

\begin{figure}[H]
\begin{center}
\includegraphics[height=2.2in]{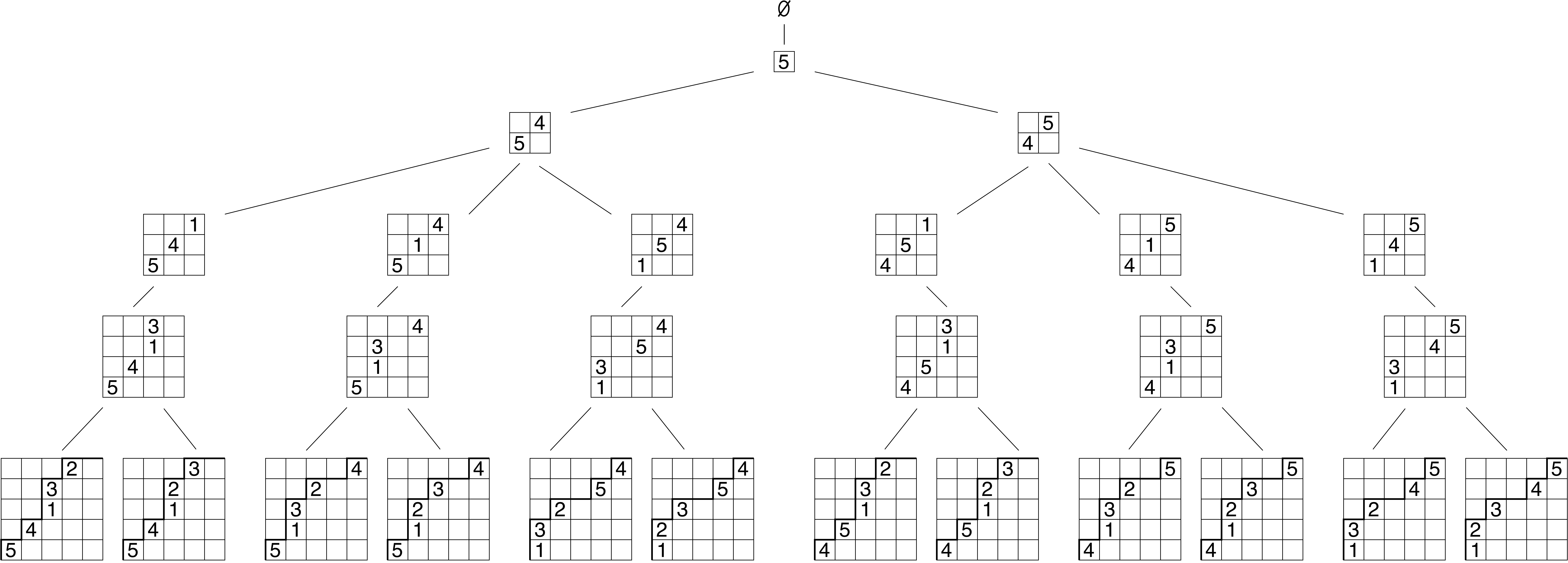}
\caption{All preference functions with diagonal word $2 \, 3 \, 1 \, 4 \, 5$ and deviation $0$.}
\label{fig:PFferm}
\end{center}
\end{figure}
\vspace{-.5cm}

Hicks \cite{Athesis} introduced a visualization of this as a tree. In Figure \ref{fig:PFferm}, we show how parking functions with diagonal word $2 \, 3 \, 1 \, 4 \, 5$ are built by inserting. The schedule numbers of $\tau$ are $(1,2,3,1,2)$. Note that at each level of the tree, the degree of each node is the schedule number corresponding to the car being inserted. Furthermore, the children of each node are arranged so that, from left to right, the change in $\dinv$ between parent and child is $0$, $1$, \dots, $w_i-1$. This is essentially the proof of the following theorem, which is due to Haglund and Loehr \citep{HL05}.

\begin{theorem}[Haglund-Loehr] \label{fermthm}
Let $\tau \in S_n$ with schedule $(w_i)$. Then
$$
\sum_{\substack{PF \in {\cal PF}_n \\ \diagword(PF)=\tau}} t^{\area(PF)} q^{\dinv(PF)} = t^{\maj(\tau)} \prod_{i=1}^n [ w_i ]_q.
$$
\end{theorem}

We extend the notion of schedules to preference functions as follows. Suppose $l \geq 0$ and $\tau \in S_n$ with at least $l+1$ runs. Let $1 \leq c \leq n$. If $c$ is in one of the last $l$ runs of $\tau$, then define $w^{(l)}(c)$ to be the number of elements smaller than $c$ in its own run plus the number of elements larger than $c$ in the previous run. If $c$ is in the \emph{$(l+1)$-st from last} run, define $w^{(l)}(c)$ to be the number of elements to the right of $c$ in the same run. Otherwise define $w^{(l)}(c)$ to be the number of elements larger than $c$ in its own run plus the number of elements smaller than $c$ in the next run.

For example, let $\tau = 2 \, 3 \, 1 \, 4 \, 5$. Then $\tau$ consists of 2 runs and we have $w^{(1)}(3) = 1$, $w^{(1)}(2)=2$, $w^{(1)}(1)=2$, $w^{(1)}(4)=1$, and $w^{(1)}(5)=2$. We say that $(w^{(l)}(c))$ are the $l$-schedule numbers of $\tau$. It is easy to see that the original schedule numbers $(w_i)$ correspond to the $0$-schedule numbers of $\tau$, but they appear in a different order. We will use the new schedule numbers $( w^{(l)}(c) )$ to build preference functions with diagonal word $\tau$ and deviation $l$. See Figure \ref{fig:Prferm} for the tree whose leaves are preference functions with diagonal word $2 \, 3 \, 1 \, 4 \, 5$ and deviation $l=1$. Note that $w^{(1)}(c)$ gives degrees of the nodes when car $c$ is inserted.

\begin{figure}[H]
\begin{center}
\includegraphics[height=2.2in]{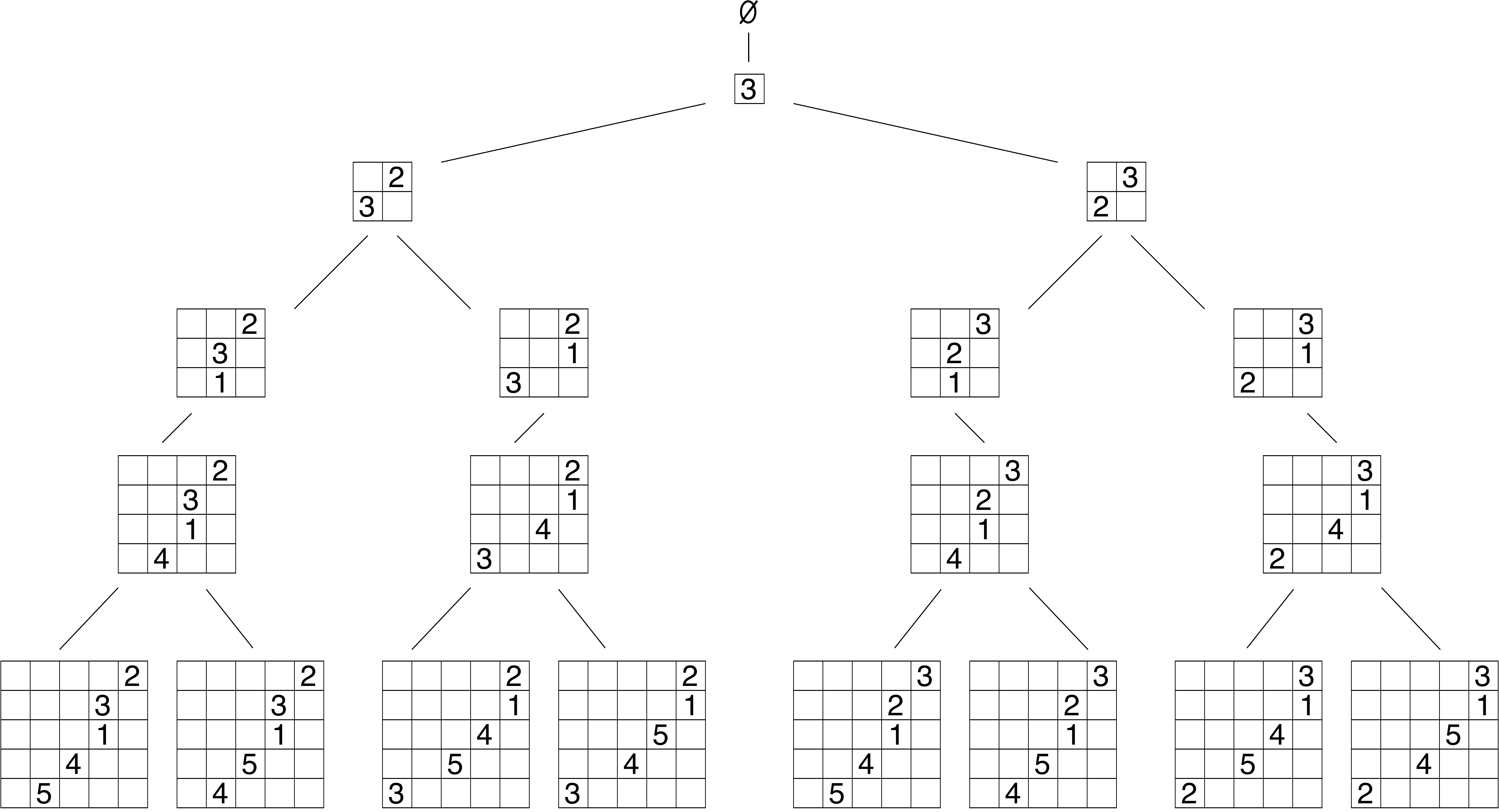}
\caption{All preference functions with diagonal word $2 \, 3 \, 1 \, 4 \, 5$ and deviation $1$.}
\label{fig:Prferm}
\end{center}
\end{figure}
\vspace{-.5cm}

\begin{theorem} \label{newferm}
Let $\tau \in S_n$ with runs of lengths $\rho_k, \dots, \rho_1, \rho_0$. Let $0 \leq l \leq k$.
$$
\sum_{\substack{Pr \in {\cal P}ref_n \\ \diagword(Pr)=\tau \\ l(Pr)=l}} t^{\area(Pr)} q^{\dinv(Pr)}
= t^{\maj(\tau)} q^{\rho_0 + \dots + \rho_{l-1}} \prod_{c=1}^n [ w^{(l)}(c) ]_q.
$$
\end{theorem}

\begin{proof}

Each element in the \emph{$i$-th from last} run of $\tau$ will contribute $i-1$ to $\area$. Therefore the factor $t^{\maj(\tau)}$ on the right hand side of Theorem \ref{newferm} accounts for the $\area$ on the left hand side. It remains to enumerate the desired preference functions by $\dinv$.

To do this, first insert each car $c$ which occurs in the first $k+1-l$ runs of $\tau$ from right to left starting in diagonal $0$ and moving up a diagonal between runs. At each step, we will have $w^{(l)}(c)$ choices which, when ordered from right to left, will contribute $0, 1, \dots, w^{(l)}(c)-1$ to primary and secondary $\dinv$. Since these cars belong to nonnegative diagonals, they contribute nothing to the tertiary $\dinv$.

Next, insert the cars of the remaining $l$ runs from left to right starting in diagonal $-1$ and moving into the next lowest diagonal at the start of each new run. Such a car $c$ can either appear directly below a larger car from the previous run (i.e., an element from the next highest diagonal of $\tau$) or directly left of a (previously inserted, hence smaller) car in the same run (i.e., same diagonal). Therefore we have $w^{(l)}(c)$ choices. These choices, when ordered from left to right, will contribute $0, 1, \dots, w^{(l)}(c)-1$ to primary and secondary $\dinv$.

Since these cars appear below diagonal 0, they also contribute to tertiary dinv. There are $\rho_0 + \dots + \rho_{l-1}$ such cars, so the tertiary $\dinv$ ``factors out,'' just as $\area$ did. And, as we observed above, each car contributes $[ w^{(l)}(c) ]_q$ to the enumeration of primary and secondary $\dinv$.
\end{proof}


\section{Shifting diagonals and schedules} \label{sec:shift}

This section is devoted to proving the following general result about preference functions.
\begin{theorem} \label{newschedthm}
Let $\tau \in S_n$ with schedule $(w_i)$. Suppose that the runs of $\tau$ have lengths $\rho_r, \dots, \rho_1, \rho_0$. If $1 \leq l \leq r$, then the multi-set of $l$-schedule numbers of $\tau$ is equal to $\{ w_i : 1 \leq i \leq n \} \cup \{ \rho_l \} \setminus \{ \rho_0 \} $.
Hence
$$
\sum_{\substack{Pr \in {\cal P}ref_n \\ \diagword(Pr)=\tau \\ l(Pr)=l}} t^{\area(Pr)} q^{\dinv(Pr)} = t^{\maj(\tau)} q^{\rho_0 + \dots + \rho_{l-1}} \frac{[\rho_l]_q}{[\rho_0]_q} \prod_{i=1}^{n} [w_i]_q.
$$
\end{theorem}

Our proof of this theorem requires a surprising lemma regarding partitions. See Figure \ref{fig:parts} for an illustration of the lemma applied to $\lambda = (3,3,2,1,0)$ with $a=4$ and $b=5$.

\begin{figure}[H]
\begin{center}
\includegraphics[width=3in]{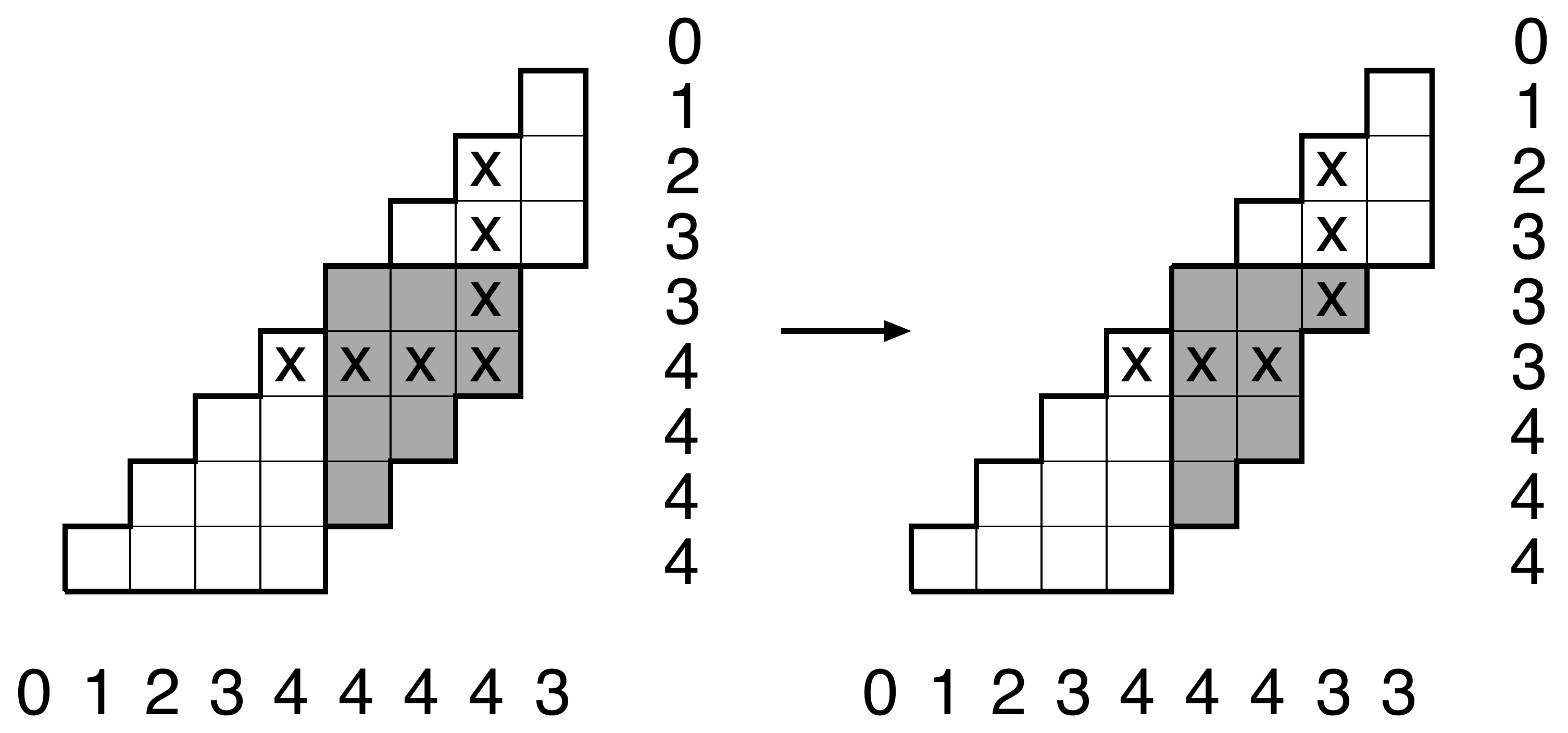}
\caption{Diagrams of $( (3,3,2,1,0) + \delta_5 ) \cup \delta_4$ and $( (3,2,2,1,0) + \delta_5 ) \cup \delta_4$.}
\label{fig:parts}
\end{center}
\end{figure}
\vspace{-.5cm}

\begin{lemma} \label{parlem}
Let $a,b > 0$ and let $\lambda = (\lambda_1 \geq \lambda_2 \geq \dots \geq \lambda_b)$ be a partition, with nonnegative parts, contained in the rectangle $a \times b$. That is $\lambda_1 \leq a$ and $l(\lambda)=b$.
We will write $\lambda'$ for the conjugate of $\lambda$ considered as a partition in the $b \times a$ rectangle. We also write $\delta_n$ for the sequence $(0,1,\dots,n-1)$ for all $n \in \mathbb{N}$. Then the sequences
$$
( \lambda + \delta_b) \cup \delta_a \hbox{ and } ( \lambda' + \delta_a ) \cup \delta_b
$$
have the same multi-set of entries. Here the sum of sequences is coordinate-wise and $\cup$ denotes concatenation.
\end{lemma}

\begin{proof}
Note that the claim holds if $\lambda$ is the empty partition. So let $\emptyset \neq \lambda$ be contained in the rectangle $a \times b$ and suppose the claim holds for all partitions contained in $\lambda$ (with $b$ parts). Suppose $\lambda_1$ occurs $k$ times in $\lambda$. Then the $k$-th entry of $\lambda + \delta_b$ is $\lambda_1 + k - 1$. Furthermore, the $\lambda_1$-st entry of $\lambda' + \delta_a$ is $k + \lambda_1 - 1$. For example, in Figure \ref{fig:parts}, $k=2$ and $\lambda_1=3$, so the marked row corresponds to the $k$-th entry of $\lambda + \delta_b$ and the marked column corresponds to the $\lambda_1$-st entry of $\lambda' + \delta_b$, and they have equal length.

Let $\mu$ be the partition obtained from $\lambda$ by reducing its $k$-th entry from $\lambda_1$ to $\lambda_1-1$. E.g., if $\lambda = (3,3,2,1,0)$ then $\mu = (3,2,2,1,0)$. The entries of $\mu + \delta_b$ are identical to the entries of $\lambda + \delta_b$ \emph{except} that the $k$-th entry is now $(\lambda_1-1) + k-1$. Similarly, the only entry of $\mu'+\delta_a$ which differs from $\lambda'+\delta_a$ is the $\lambda_1$-st entry, which is now $(k-1) + \lambda_1 -1$.

For any sequence $\sigma$, let $\{\sigma\}$ denote the multi-set of $\sigma$'s entries. Then
$$
\{ ( \lambda + \delta_b ) \cup \delta_a \} = \{ (\mu + \delta_b) \cup \delta_a \} \cup \{ \lambda_1 + k - 1\} \setminus \{ (\lambda_1 - 1) + k - 1 \}
$$
and
$$
\{ ( \lambda' + \delta_a ) \cup \delta_b \} = \{ (\mu' + \delta_a) \cup \delta_b \} \cup \{ k + \lambda_1 - 1\} \setminus \{ (k-1) + \lambda_1 - 1 \}.
$$
Since the claim holds for $\mu$, it also holds for $\lambda$. By induction, it holds for all partitions.
\end{proof}

In Figure \ref{fig:parts} we can see the geometric intuition behind our proof of the Lemma. Namely, the marked corner lies in a row and a column of equal length. In fact \emph{all} removable corners of $\lambda$ lie in equal rows and columns. Hence removing any one of them preserves the correspondence between row parts (i.e., $(\lambda + \delta_b) \cup \delta_a$) and column parts (i.e., $(\lambda'+\delta_a) \cup \delta_b$).

\begin{proof}[Proof of Theorem \ref{newschedthm}]
 
We claim that
\begin{equation} \label{eq:ws}
\{ w^{(l-1)}(c) : 1 \leq c \leq n\} \setminus \{ \rho_{l-1} \} = \{ w^{(l)}(c) : 1 \leq c \leq n\} \setminus \{ \rho_{l} \}.
\end{equation}
as multi-sets for all $1 \leq l \leq r$. Note that if $c$ is the leftmost element of the \emph{$(m+1)$-st from last} run, then $w^{(m)}(c) = \rho_m$, hence there is no trouble with the multi-set subtractions above. Once (\ref{eq:ws}) is shown, we will have
$$
\{ w^{(0)}(c) : 1 \leq c \leq n\} \setminus \{ \rho_0 \} =  \{ w^{(m)}(c) : 1 \leq c \leq n\} \setminus \{ \rho_{m} \}.
$$
for each $1 \leq m \leq r$, which is equivalent to the desired formula.

Let $1 \leq l \leq r$. Note that $w^{(l-1)}(c) = w^{(l)}(c)$ unless $c$ is in the \emph{$l$-th or $(l+1)$-st from last} run of $\tau$. This is because the calculation of a schedule number depends only on its place $\tau$ and whether the car in question lies in a positive, zero, or negative diagonal. Shifting the deviation by one only changes the positive/zero/negative ``status" of cars from two runs. For example, consider the case $\tau = 3 \, 7 \, 1 \, 5 \, 8 \, 2 \, 6 \, 4$ with $l=1,2,3$.
\begin{align*}
        c \,\, = & \hspace{.1in} 3 \, 7 \hspace{.1in} 1 \, 5 \, 8 \hspace{.1in} 2 \, 6 \hspace{.1in} 4 \\
w^{(0)}(c) = & \hspace{.1in} 2 \, 2 \hspace{.1in} 2 \, 2 \, 2 \hspace{.1in} 1 \, 1 \hspace{.1in} 1 \\
w^{(1)}(c) = & \hspace{.1in} 2 \, 2 \hspace{.1in} 2 \, 2 \, 2 \hspace{.1in} 2 \, 1 \hspace{.1in} 1 \\
w^{(2)}(c) = & \hspace{.1in} 2 \, 2 \hspace{.1in} 3 \, 2 \, 1 \hspace{.1in} 2 \, 2 \hspace{.1in} 1 \\
w^{(3)}(c) = & \hspace{.1in} 2 \, 1 \hspace{.1in} 2 \, 2 \, 2 \hspace{.1in} 2 \, 2 \hspace{.1in} 1
\end{align*}
We can see here that schedule numbers only change within two runs of $\tau$ whenever we shift $l$. Therefore it is sufficient to prove our claim for $\tau$ with a single descent and $l=1$ (that is, for the case when the preference functions in question are contained in two diagonals).

Suppose $\tau \in S_n$ with a single descent. For a finite set $A$, let $A^\uparrow$ denote the word consisting of the elements of $A$ in increasing order. Then $\tau = B^\uparrow \, A^\uparrow$ for some disjoint $A,B$. Let $\lambda \subseteq |A| \times |B|$ be the partition whose $i$th part is the number of elements of $A$ which are smaller than the $i$-th largest element of $B$. Then $\lambda'$ is the partition whose $j$th part is the number of elements of $B$ which are larger than the $j$-th smallest element of $A$.

Let $w^{(l)}_i = w^{(l)}(c)$ for $c= \tau_{n+1-i}$. Then for $i$ from $1$ to $|A|$, $w^{(0)}_i = i$, and for $j$ from $1$ to $|B|$, $w^{(0)}_{|A|+j} =  \lambda_j + j-1$. Hence the $0$-schedule numbers of $\tau$ form the multi-set $\{ ( \lambda + \delta_{|B|}) \cup \delta_{|A|} \} \cup \{|A|\} \setminus \{0\} $. On the other hand, for $i$ from $1$ to $|A|$, $w^{(1)}_i = \lambda'_{|A|-i+1} + |A|-i$, and for $j$ from $1$ to $|B|$, $w^{(1)}_{|A|+j} = j$. Then the $1$-schedule numbers of $\tau$ form $\{ ( \lambda' + \delta_{|A|}) \cup \delta_{|B|} \} \cup \{|B|\} \setminus \{0\}$.

For example, consider $\tau = 3 \, 4 \, 5 \, 8 \, 1 \, 2 \, 6 \, 7 \, 9$. Then $A = \{1,2,6,7,9\}$ and $B = \{3,4,5,8\}$. This gives $\lambda = (4,2,2,2)$ and $\lambda' = (4,4,1,1,0)$. Furthermore, we have
\begin{align*}
c \, \, = & \hspace{.1in} 3 \, 4 \, 5 \, 8 \hspace{.1in} 1 \, 2 \, 6 \, 7 \, 9\\
w^{(0)}(c) = & \hspace{.1in} 5 \, 4 \, 3 \, 4 \hspace{.1in} 5 \, 4 \, 3 \, 2 \, 1\\
= & \hspace{.08in} \substack{3 \\ + \\ 2} \, \substack{2 \\ + \\ 2} \, \substack{1 \\ + \\ 2} \, \substack{0 \\ + \\ 4} \hspace{.08in} 5 \, 4 \, 3 \, 2 \, 1\\
w^{(1)}(c) = & \hspace{.1in} 4 \, 3 \, 2 \, 1 \hspace{.1in} 4 \, 5 \, 3 \, 4 \, 4\\
= & \hspace{.1in} 4 \, 3 \, 2 \, 1 \hspace{.08in} \substack{ 0 \\ + \\ 4} \, \substack{1 \\ + \\ 4} \, \substack{2 \\ + \\ 1} \, \substack{3 \\ + \\ 1} \, \substack{4 \\ + \\ 0}
\end{align*}

If we remove a single copy of $\rho_0 = |A|$ from $\{ w^{(0)}_i \}$ and a single copy of $\rho_1=|B|$ from $\{ w^{(1)}_i \}$ and insert the missing $0$'s, then Lemma \ref{parlem} applies. Hence $\{ w^{(0)}_i \} \setminus \{ \rho_{0} \} = \{ w^{(1)}_i \} \setminus \{ \rho_{1} \}$ as desired.
\end{proof}

\begin{corollary} \label{noides}
Let $\tau \in S_n$ with schedule $(w_i)$ and let $k$ be the length of its last run. We have
$$
\sum_{\substack{Pr \in {\cal P}ref_n \\ \diagword(Pr)=\tau}} t^{\area(Pr)} q^{\dinv(Pr)}
= t^{\maj(\tau)} \frac{[n]_q}{[k]_q} \prod_{i=1}^{n} [w_i]_q
= \frac{[n]_q}{[k]_q} \sum_{\substack{PF \in {\cal PF}_n \\ \diagword(PF)=\tau}} t^{\area(PF)} q^{\dinv(PF)}.
$$
\end{corollary}

\begin{proof}
We simply note that if $\tau$'s runs are given by $\rho_r, \dots, \rho_1, \rho_0$ (so that $\rho_0 + \dots + \rho_r=n$ and $\rho_0=k$), then
\begin{align*}
\sum_{\diagword(Pr)=\tau} t^{\area(Pr)} q^{\dinv(Pr)} &= \sum_{l=0}^r \left( \sum_{\substack{\diagword(Pr)=\tau \\ l(Pr)=l}} t^{\area(Pr)} q^{\dinv(Pr)} \right) \\
&= t^{\maj(\tau)} \frac{1}{[\rho_0]_q} \left( \sum_{l=0}^{r} q^{\rho_0 + \dots + \rho_{l-1} } [\rho_l]_q \right)  \prod_{i=1}^n [w_i]_q \\
&= t^{\maj(\tau)} \frac{[n]_q}{[k]_q} \prod_{i=1}^n [w_i]_q.
\end{align*}
This gives the first equality. To obtain the second, apply Theorem \ref{fermthm}.
\end{proof}


\pagebreak

\section{Dealing with Inverse Descents} \label{sec:ides}

In order to address the Square Paths Conjecture, we need to enumerate preference functions by $\area$, $\dinv$ \emph{and} $\ides$. In her thesis, Hicks \citep{Athesis} shows that the $\ides$ ``factors out'' of the desired enumeration for parking functions. We follow her notation here and prove the corresponding result for preference functions.

For any permutation $\tau$, we can partition the set $\{1,2,\dots,n\}$ according to whether $i$ appears directly left of $i+1$ in $\tau$. Call each such part a \emph{consecutive block} of $\tau$. E.g., the consecutive blocks of $\tau = 8 \, 9 \, 5 \, 4 \, 6 \, 7 \, 1 \, 2 \, 3$ are $\{8,9\}$, $\{5\}$, $\{4\}$, $\{6,7\}$, $\{1,2,3\}$. Let $\Yconsec(\tau)$ be the Young subgroup of $S_n$ which permutes elements in the same consecutive block of $\tau$. In the example, $\Yconsec(\tau) = S_{\{1,2,3\}} \times S_{\{4\}} \times S_{\{5\}} \times S_{\{6,7\}} \times S_{\{8,9\}}$.

\begin{lemma} \label{factorlemma}
Let $l \geq 0$. Suppose $\tau \in S_n$ has at least $l+1$ runs. Then
\begin{align*}
\sum_{\substack{Pr \in {\cal P}ref_n \\ \diagword(Pr)=\tau \\ l(Pr)=l}} &t^{\area(Pr)} q^{\dinv(Pr)} Q_{\ides(Pr)}\\
&= \left(\sum_{\substack{Pr \in {\cal P}ref_n \\ \diagword(Pr)=\tau \\ l(Pr)=l}} t^{\area(Pr)} q^{\dinv(Pr)}\right) \cdot \left( \frac{ \displaystyle \sum_{\pi \in \Yconsec(\tau)} q^{\inv(\pi)} Q_{\ides(\tau) \cup \ides(\pi)} }{ \displaystyle \sum_{\pi \in \Yconsec(\tau)} q^{\inv(\pi)} } \right).
\end{align*}
\end{lemma}

\noindent The case $l=0$ of this lemma is equivalent to Corollary 74 of \cite{Athesis}. Its proof extends without issue to this more general setting. However, for the sake of completeness, we provide a sketch of this proof below.

\begin{proof}[Proof Sketch]

Let ${\cal P}ref_{\tau,l}$ be the set of preference functions with diagonal word $\tau$ and deviation $l$. Note that $\ides(\tau) \subseteq \ides(Pr)$. This is because $i \in \ides(\tau)$ iff $i+1$ occurs in a higher diagonal of $Pr$ than $i$, which means that $i+1$ will precede $i$ in $\sigma(Pr)$. Any other element of $\ides(Pr)$ corresponds to some $i$ and $i+1$ in the same consecutive block of $\tau$. Hence, each $Pr \in {\cal P}ref_{\tau,l}$ can be uniquely decomposed into a pair consisting of another preference function $Pr' \in {\cal P}ref_{\tau,l}$ with $\ides(Pr')=\ides(\tau)$ and a permutation $\pi \in \Yconsec(\tau)$ so that if we permute the cars of $Pr'$ according to $\pi$, we obtain $Pr$.

\begin{figure}[H]
\begin{center}
\hspace{1in} \includegraphics[width=4in]{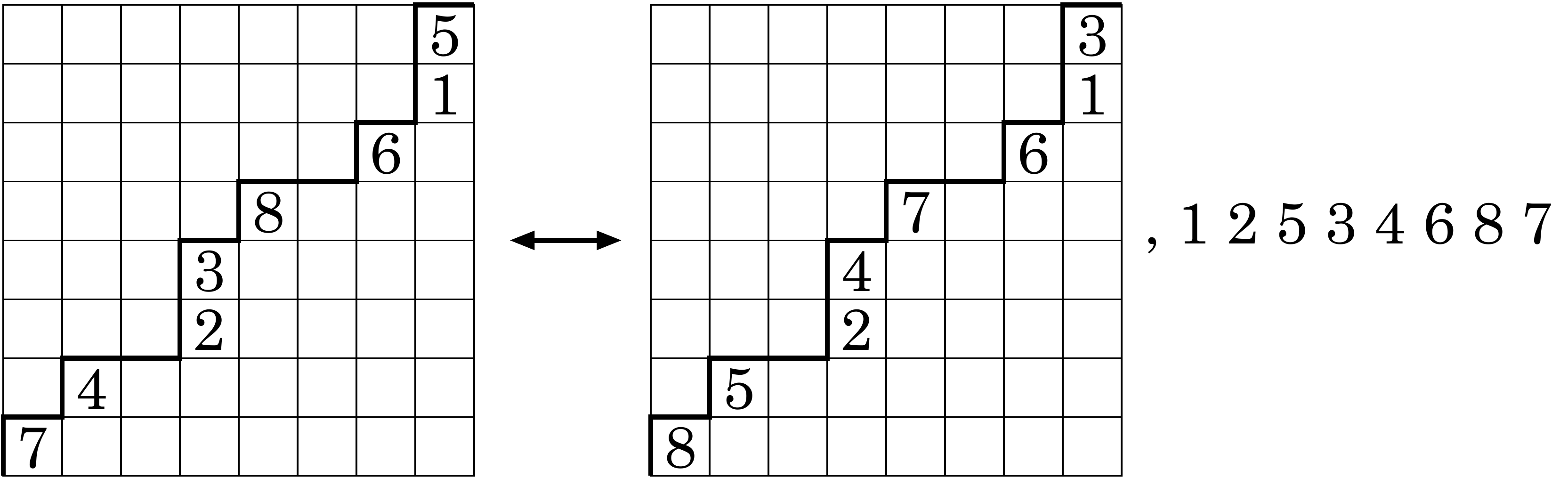}
\caption{A decomposition of a preference function according to consecutive blocks.}
\label{fig:Yconsec}
\end{center}
\end{figure}
\vspace{-.5cm}

For example, consider Figure \ref{fig:Yconsec}. On the left side of the figure, we have a preference function $Pr$ with diagonal word $\tau = 3 \, 4 \, 5 \, 7 \, 8 \, 1 \, 2 \, 6$ and deviation $l=1$. Furthermore $\ides(Pr) = \{2,4,6,7\}$ and $\ides(\tau) = \{2,6\}$. On the right we have a preference function $Pr'$ with $\ides(Pr')=\ides(\tau)$ and a permutation $\pi$ consisting of a cycle on $\{3,4,5\}$ and a transposition on $\{7,8\}$. The consecutive blocks of $\tau$ are $\{1,2\}, \{3,4,5\}, \{6\}, \{7,8\}$, so $\pi \in Yconsec(\tau)$.

In general, we have that $\ides(Pr) = \ides(\tau) \cup \ides(\pi)$ and $\dinv(Pr) = \dinv(Pr') + \inv(\pi)$. Note that $Pr$ and $Pr'$ have identical $\dinv$ pairs and cars below the diagonal with one exception. $Pr$ contains primary $\dinv$ between consecutive cars and $Pr'$ does not. But $\pi$ encodes the way that consecutive cars within a diagonal (within a single consecutive block of $\tau$) are interleaved and hence how many primary $\dinv$s occur between them. Similarly, $Pr$ and $Pr'$ share $\ides$ except those caused by pairs $i$ and $i+1$ in the same diagonal, which are recorded by $\pi$.

Let ${\cal P}ref_{\tau,l}^{id}$ be the set of preference functions $Pr \in {\cal P}ref_{\tau,l}$ which corresponds to itself and the identity permutation under this decomposition. Then we have
$$
\sum_{Pr \in {\cal P}ref_{\tau,l}} t^{\area(Pr)} q^{\dinv(Pr)}
= \left( \sum_{Pr' \in {\cal P}ref_{\tau,l}^{id}} t^{\area(Pr')} q^{\dinv(Pr')} \right) \cdot
\left( \sum_{\pi \in \Yconsec(\tau)} q^{\inv(\pi)} \right)
$$
and
\begin{align*}
\sum_{Pr \in {\cal P}ref_{\tau,l}} &t^{\area(Pr)} q^{\dinv(Pr)} Q_{\ides(Pr)}\\
&= \left( \sum_{Pr' \in {\cal P}ref_{\tau,l}^{id}} t^{\area(Pr')} q^{\dinv(Pr')} \right) \cdot
\left( \sum_{\pi \in \Yconsec(\tau)} q^{\inv(\pi)} Q_{\ides(\tau) \cup \ides(\pi)} \right).
\end{align*}
\noindent Combining these equations gives the desired result.
\end{proof}

Fixing $\tau$, if we sum Lemma \ref{factorlemma} over $l$ and compare with the case $l=0$, we see that the $\ides$-less enumerations of preference functions and parking functions differ from the full enumeration by the same factor. This fact, combined with Corollary \ref{noides} gives the following.
\begin{corollary} \label{withides}
Let $\tau \in S_n$ and let $k$ be the length of its last run. Then
$$
\sum_{\substack{Pr \in {\cal P}ref_n \\ \diagword(Pr)=\tau}} t^{\area(Pr)} q^{\dinv(Pr)} Q_{\ides(Pr)}
= \frac{[n]_q}{[k]_q} \sum_{\substack{PF \in {\cal PF}_n \\ \diagword(PF)=\tau}} t^{\area(PF)} q^{\dinv(PF)} Q_{\ides(PF)}.
$$
\end{corollary}
\noindent Now we can relate the right hand side of this equation to $\nabla$ using a corollary of the Compositional Shuffle Conjecture. More precisely, in \citep{compconj}, Haglund, Morse and Zabrocki refined the Shuffle Conjecture using the following plethystic symmetric function operators.
$$
C_a \, P[X] \, = \, \left(- \frac{1}{q} \right)^{a-1} P \hspace{-.1cm} \left[ X - \frac{1-1/q}{z} \right] \, \sum_{m \geq 0} z^m h_m[X] \, \Big|_{z^a}
$$
Their conjecture, which is stated below, was recently proved by Carlsson and Mellit in \citep{compproof}. Here $\comp(PF)$ is the composition of $n$ giving the distances between points $(i,i)$ on $PF$'s underlying path. For example, the parking function in Figure \ref{fig:pref} has $\comp = (4,1)$. 
\begin{theorem}[Carlsson-Mellit]  \label{thm:comp}
For all compositions $\rho \models n$,
$$
C_{\rho_1} \cdots C_{\rho_k} 1 = \sum_{\substack{ PF \in {\cal PF}_n \\ \comp(PF)=\rho}} t^{\area(PF)} q^{\dinv(PF)} Q_{\ides(PF)}.
$$
\end{theorem}

Let $PF$ be a parking function. Define $\touch(PF)$ to be the number of parts of $\comp(PF)$, i.e., the number of cars in the main diagonal $y=x$. For $n \in \mathbb{N}$ and $1 \leq k \leq n$, Garsia and Haglund \citep{qtCat} define symmetric functions $E_{n,k}$ so that
$$
e_n \left[ X \frac{1-z}{1-q} \right] = \sum_{k=1}^n \frac{(z;q)_k}{(q;q)_k} E_{n,k}[X].
$$
where
$$
(z;q)_n = (1-z)(1-zq) \cdots (1-zq^{n-1}).
$$
Haglund, Morse and Zabrocki \citep{compconj} showed
\begin{theorem}[Haglund-Morse-Zabrocki]
For all $1 \leq k \leq n$,
$$
E_{n,k} = \sum_{\rho \models n, \, \, l(\rho)=k} C_{\rho_1} \cdots C_{\rho_k} 1.
$$
\end{theorem}
\noindent Hence Theorem \ref{thm:comp} implies
\begin{corollary} \label{thm:touch}
For all $n \in \mathbb{N}$ and all $1 \leq k \leq n$,
$$
\nabla E_{n,k} = \sum_{\substack{PF \in {\cal PF}_n \\ \touch(PF)=k}} t^{\area(PF)} q^{\dinv(PF)} Q_{\ides(PF)}.
$$
\end{corollary}
\noindent Then summing Corollary \ref{withides} over all $\tau$ whose last run has length $k$ and applying Corollary \ref{thm:touch} gives
\begin{theorem} \label{thm:touchpref}
For all $1 \leq k \leq n$,
$$
\sum_{\substack{Pr \in {\cal P}ref_n \\ \touch(Pr)=k}} t^{\area(Pr)} q^{\dinv(Pr)} Q_{\ides(Pr)} 
= \frac{[n]_q}{[k]_q} \nabla E_{n,k}
$$
where $\touch(Pr)$ is the number of cars in diagonal $-l(Pr)$ for any preference function $Pr$. (It is also the length of the last run of $\diagword(Pr)$.)
\end{theorem}

Finally, we need a symmetric function identity relating $p_n$ to the polynomials $\{ E_{n,k} \}$. The following identity was proved by Can and Loehr \citep{sqsignproof} in their proof of a special case of the Square Paths Conjecture. It seems this was known earlier to Garsia and Haglund \citep{qtCat}.
\begin{theorem}[Garsia-Haglund] \label{Esymmid}
For all $n \geq 1$,
$$
(-1)^{n-1} p_n = \sum_{k=1}^n \frac{[n]_q}{[k]_q} E_{n,k}.
$$
\end{theorem}
\noindent Hence summing Theorem \ref{thm:touchpref} over $k$ and applying Theorem \ref{Esymmid} gives the Square Paths Conjecture.
\begin{theorem}
For all $n \geq 1$,
$$
\sum_{Pr \in {\cal P}ref_n} t^{\area(Pr)} q^{\dinv(Pr)} Q_{\ides(Pr)} = (-1)^{n-1} \nabla p_n.
$$
\end{theorem}


\nocite{*}
\bibliographystyle{plainnat} 
\bibliography{nablapn}

\begin{thebibliography}{15}
\providecommand{\natexlab}[1]{#1}
\providecommand{\url}[1]{\texttt{#1}}
\expandafter\ifx\csname urlstyle\endcsname\relax
  \providecommand{\doi}[1]{doi: #1}\else
  \providecommand{\doi}{doi: \begingroup \urlstyle{rm}\Url}\fi

\bibitem[Bergeron and Garsia(1999)]{SciFi}
F.~Bergeron and A.~M. Garsia.
\newblock {Science Fiction and Macdonald's Polynomials}.
\newblock \emph{CRM Proc. \& Lecture Notes, Amer. Math. Soc.}, 22:\penalty0
  1--52, 1999.

\bibitem[Can and Loehr(2006)]{sqsignproof}
M.~Can and N.~Loehr.
\newblock A proof of the $q,t$-square conjecture.
\newblock \emph{J. Comb. Theory Series A}, 113.7:\penalty0 1419--1434, 2006.

\bibitem[Carlsson and Mellit(2015)]{compproof}
E.~Carlsson and A.~Mellit.
\newblock A proof of the shuffle conjecture.
\newblock \emph{arXiv preprint arXiv:1508.06239}, 2015.

\bibitem[Garsia and Haglund(2002)]{qtCat}
A.~M. Garsia and J.~Haglund.
\newblock {A proof of the $q,t$-Catalan positivity conjecture}.
\newblock \emph{Discrete Math.}, 256:\penalty0 677--717, 2002.

\bibitem[Garsia and Haiman(1996)]{modmac}
A.~M. Garsia and M.~Haiman.
\newblock {Some Natural Bigraded $S_n$-Modules and $q,t$-Kostka Coefficients}.
\newblock \emph{Electron. J. Combin.}, 3\penalty0 (2), 1996.

\bibitem[Gessel(1984)]{Gessel}
I.~Gessel.
\newblock {Multipartite P-partitions and inner products of skew Schur
  functions}.
\newblock \emph{Contemp. Math.}, 34:\penalty0 289--301, 1984.

\bibitem[Haglund and Loehr(2005)]{HL05}
J.~Haglund and N.~Loehr.
\newblock {A conjectured combinatorial formula for the Hilbert series for
  diagonal harmonics}.
\newblock \emph{Discrete Math.}, 298\penalty0 (1):\penalty0 189--204, 2005.

\bibitem[Haglund et~al.(2005)Haglund, Haiman, Loehr, Remmel, and
  Ulyanov]{shuffleconj}
J.~Haglund, M.~Haiman, N.~Loehr, J.~B. Remmel, and A.~Ulyanov.
\newblock {A combinatorial formula for the character of the diagonal
  coinvariants}.
\newblock \emph{Duke J. Math.}, 126:\penalty0 195--232, 2005.

\bibitem[Haglund et~al.(2012)Haglund, Morse, and Zabrocki]{compconj}
J.~Haglund, J.~Morse, and M.~Zabrocki.
\newblock {A compositional refinement of the shuffle conjecture specifying
  touch points of the Dyck path}.
\newblock \emph{Canad. J. Math.}, 64:\penalty0 822--844, 2012.

\bibitem[Haiman(2001)]{frobchar}
M.~Haiman.
\newblock {Hilbert schemes, polygraphs and the Macdonald positivity
  conjecture}.
\newblock \emph{J. Amer. Math. Soc.}, 14\penalty0 (4):\penalty0 941--1006,
  2001.

\bibitem[Hicks(2013)]{Athesis}
A.~Hicks.
\newblock \emph{{Parking Function Polynomials and Their Relation to the Shuffle
  Conjecture}}.
\newblock PhD thesis, University of California, San Diego, 2013.

\bibitem[Konheim and Weiss(1966)]{KW}
A.~G. Konheim and B.~Weiss.
\newblock An occupancy discipline and applications.
\newblock \emph{SIAM J. Appl. Math.}, 14\penalty0 (6):\penalty0 1266--1274,
  1966.

\bibitem[Loehr and Warrington(2007)]{sqconj}
N.~Loehr and G.~Warrington.
\newblock Square $q,t$-lattice paths and $\nabla(p_n)$.
\newblock \emph{Trans. Amer. Math. Soc.}, 359\penalty0 (2):\penalty0 649--669,
  2007.

\bibitem[Macdonald(1988)]{macoriginal}
I.~G. Macdonald.
\newblock {A new class of symmetric functions}.
\newblock \emph{Actes du 20e S´eminaire Lotharingien}, Publ. I.R.M.A.
  Strasbourg:\penalty0 131--171, 1988.

\bibitem[Macdonald(1995)]{macbook}
I.~G. Macdonald.
\newblock \emph{Symmetric functions and Hall polynomials}.
\newblock Oxford Mathematical Monographs, New York, 2nd edition, 1995.

\end{thebibliography}

\end{document}